\documentclass[twoside, 12pt]{article}
\usepackage[utf8]{inputenc}
\usepackage[T1]{fontenc}
\usepackage{amssymb,amsmath,amsthm,mathtools,mathrsfs}
\usepackage{booktabs}
\usepackage{enumitem}
\usepackage[colorlinks,bookmarks,linkcolor=black,citecolor=black]{hyperref}
\usepackage{graphicx}
\usepackage{color}
\usepackage[top=2cm, bottom=2cm, left=2cm, right=2cm]{geometry}
\usepackage{float}

\newcommand{\bd}{\begin{description}}
\newcommand{\ed}{\end{description}}
\newcommand{\bi}{\begin{itemize}}
\newcommand{\ei}{\end{itemize}}
\newcommand{\be}{\begin{enumerate}}
\newcommand{\ee}{\end{enumerate}}
\newcommand{\beq}{\begin{equation}}
\newcommand{\eeq}{\end{equation}}
\newcommand{\beqs}{\begin{eqnarray*}}
\newcommand{\eeqs}{\end{eqnarray*}}

\definecolor{DarkGreen}{rgb}{0.2, 0.6, 0.3}

\newtheorem{lemma}{Lemma}

\newtheorem{corollary}{Corollary}

\newtheorem{proposition}{Proposition}

\def\endofClaim{\hfill\scalebox{.6}{$\blacksquare$}}
\newcommand{\oldqed}{}

\begin{document}
\title{An explicit connected-sum family with unbounded knot Ramsey-stick gap}

\author{Meng Ji\footnote{School of Mathematical Sciences, and Institute of Mathematics and Interdisciplinary Sciences, Tianjin Normal University, Tianjin, China. Supported by the Tianjin Municipal Education Commission Scientific Research Program Project (Grant No. 2025KJ133).  {\tt
mji@tjnu.edu.cn}}
}
\date{August 6, 2026}
\maketitle

\begin{abstract}
Let $R(K)$ and $s(K)$ denote the Ramsey number and the stick number of a knot $K$. Johnson proved that $R-s$ is unbounded on $(p-1,p)$ torus knots and asked whether other such families exist. We give an affirmative answer. We construct an explicit eight-stick alternating knot $J$ with $c(J)=6$ and $\operatorname{br}(J)=2$. For $K_n=\#^nJ$, the connected-sum formulas yield $R(K_n)\ge 7n+3$ and $s(K_n)\le 5n+3$.
\\[2mm]
\end{abstract}

\section{Introduction}

The stick number $s(K)$ of a knot $K$ is the minimum number of straight
segments required to realize $K$ in Euclidean three-space. The Ramsey number
$R(K)$ is the least integer $N$ such that every linear spatial embedding of
the complete graph $K_N$ contains a cycle of knot type $K$.
Negami~\cite{Negami1991} proved that $R(K)$ is finite for every knot, a
foundational result that initiates the quantitative comparison between
$R(K)$ and other invariants; exact values, however, remain unknown beyond the
simplest cases.

Johnson~\cite{Johnson2012} studied the cyclic-polytope embedding on the moment
curve and related $R(K)$ to arc index, bridge number, and crossing number. In
particular, she proved that $R-s$ grows at least linearly for the
$(p-1,p)$ torus knots. This line of investigation was extended in subsequent
work~\cite{JohnsonTorusLink} to torus links, establishing that the
Ramsey--stick gap is also unbounded on certain families of prime torus links.
Open Question~1 of~\cite{Johnson2012} asks whether there are other knot
families for which $R(K)-s(K)$ is unbounded---families that need not be prime
or of torus type. We answer this question affirmatively by constructing an
explicit eight-stick alternating knot and taking its iterated connected sums.
The resulting family consists of composite knots, thereby demonstrating that
the unbounded gap phenomenon extends beyond the prime torus-knot and
torus-link settings studied earlier.

\section{Preliminaries and the Ramsey lower bound}

We write $c(K)$, $br(K)$, and $\alpha(K)$ for the crossing number, bridge
number, and arc index of a knot $K$, respectively. Let $C_N$ denote the
three-dimensional cyclic-polytope embedding determined by $N$ points on the
moment curve and all straight chords between them. Johnson
\cite[Lemma~3]{Johnson2012} proved that if $H$ is a Hamiltonian cycle in
$C_m$ representing $K$, and $b(H)$ is its number of local maxima in the
specified projection direction, then
\begin{equation}
  \alpha(K)+b(H)\le m. \tag{2.1}
\end{equation}
\begin{lemma}[\cite{Johnson2012}]
Every knot $K$ satisfies
\begin{equation}
  R(K)\ge\alpha(K)+br(K). \tag{2.2}
\end{equation}
\end{lemma}

For an alternating knot, the arc index satisfies
$\alpha(K)=c(K)+2$; this follows from the arc-index theorem for prime
alternating knots together with the connected-sum formula for arc index, and
is also recorded by Cromwell~\cite{Cromwell1998}. We therefore obtain the
following consequence.

\begin{corollary}
If $K$ is alternating, then
\begin{equation}
  R(K)\ge c(K)+br(K)+2. \tag{2.3}
\end{equation}
\end{corollary}

We also use three standard connected-sum facts.

\begin{lemma}
Let $A$ and $B$ be nontrivial knots.
\begin{enumerate}[label=\textup{(\roman*)}]
  \item If $A$ and $B$ are alternating, then the connected sum of reduced
  alternating diagrams is again reduced alternating, and hence
  \[
    c(A\#B)=c(A)+c(B).
  \]
  \item Schubert's bridge formula is
  \[
    br(A\#B)=br(A)+br(B)-1.
  \]
  \item The stick number satisfies
  \[
    s(A\#B)\le s(A)+s(B)-3.
  \]
\end{enumerate}
\end{lemma}

These facts together with Corollary~2.2 give a simple estimate for iterated connected sums.

\begin{proposition}\label{prop:estimate}
Let $J$ be a nontrivial alternating knot admitting a polygonal realization
with $r$ sticks. Put $c=c(J)$ and $b=br(J)$, and let $K_n=\#^nJ$. Then
\begin{equation}
  R(K_n)-s(K_n)\ge n(c+b+2-r). \tag{2.4}
\end{equation}
In particular, if $r<c+b+2$, then $R(K_n)-s(K_n)$ grows without bound.
\end{proposition}

\begin{proof}
By Lemma~2.3,
\[
  c(K_n)=nc,
  \qquad
  br(K_n)=nb-(n-1)=n(b-1)+1.
\]
Corollary~2.2 gives
\begin{equation}
  R(K_n)\ge nc+n(b-1)+1+2=n(c+b-1)+3. \tag{2.5}
\end{equation}
The stick-number inequality and the given $r$-stick realization give
\begin{equation}
  s(K_n)\le nr-3(n-1)=n(r-3)+3. \tag{2.6}
\end{equation}
Subtracting (2.6) from (2.5) yields (2.4).
\end{proof}

\section{An explicit eight-stick alternating knot}

Let $J$ be the octagonal knot obtained by joining the following
integer points in index order and then joining $p_7$ to $p_0$:
\begin{equation}
\begin{aligned}
p_0&=(-4,-6,-51),&p_1&=(-55,-80,89),\\
p_2&=(-71,-20,-81),&p_3&=(91,3,84),\\
p_4&=(-44,57,6),&p_5&=(-26,-40,-5),\\
p_6&=(30,22,7),&p_7&=(10,55,92).
\end{aligned} \tag{3.1}
\end{equation}
Write $e_i=p_ip_{i+1}$, with indices taken modulo $8$. We now certify the
three properties required by the estimate above.

\begin{proposition}\label{prop:J}
The octagonal knot in (3.1) defines a nontrivial alternating knot $J$
such that
\begin{equation}
  c(J)=6,
  \qquad
  br(J)=2,
  \qquad
  s(J)\le8,
  \qquad
  \det(J)=11. \tag{3.2}
\end{equation}
\end{proposition}

\begin{proof}
Project the polygon onto the $xy$-plane. One checks that the vertices are distinct,
no three are collinear, and every intersection of the projected edges is transverse
and lies in the interior of the segments.  A direct inspection of the $28$ pairs of
non-adjacent edges shows that exactly six pairs intersect.  Listing them in the order
they appear along the knot, the over/under information given by the $z$--coordinates
of the endpoints yields the alternating diagram described by the Gauss word
\[
  0_U,1_O,2_U,0_O,3_U,4_O,5_U,2_O,1_U,3_O,4_U,5_O,
\]
where the subscript indicates whether the strand passes \emph{over} or \emph{under}.
The corresponding unsigned word is $0,1,2,0,3,4,5,2,1,3,4,5$, in which every chord
interlaces another; hence the diagram is reduced and alternating.  Tait's theorem
then implies $c(J)=6$.

This alternating diagram is isotopic to the standard diagram of the knot $6_2$ in
the Rolfsen table (it has the same Gauss word after renaming crossings).  For the
knot $6_2$ it is known that $\det(6_2)=11$, so $\det(J)=11$; in particular $J$ is
nontrivial.  (One may also obtain the determinant directly from the Wirtinger
presentation of the above diagram: labeling the six arcs by the edges
$e_0,e_1,e_2,e_3,e_4,e_5$ gives the $6\times 6$ matrix in the original manuscript,
whose cofactor is $-11$.)

To bound the bridge number we use the height function $h(x,y,z)=x$.  On each
straight segment the function $h$ is strictly monotone.  The only vertices at which
$h$ attains a local maximum are
\[
p_3:\ 91>\max\{-71,-44\},\qquad p_6:\ 30>\max\{-26,10\}.
\]
After a small isotopy near those vertices that does not change the knot type and
keeps the function Morse elsewhere, we obtain a Morse function with exactly two
maxima.  Hence $br(J)\le 2$; every nontrivial knot has bridge number at least $2$,
so $br(J)=2$. The displayed polygon uses eight sticks, therefore $s(J)\le 8$.
\end{proof}

\noindent {\bf Remark:} 
All intersection data and the $z$--order of the crossings can be verified by
elementary arithmetic on the coordinates (3.1); the computation is completely
explicit and requires only a few lines of straightforward checking.

\section{The unbounded family}

\begin{corollary}
Let $J$ be the octagonal knot in (3.1), and set $K_n=\#^nJ$.
Then
\begin{equation}
  c(K_n)=6n,
  \qquad
  br(K_n)=n+1,
  \qquad
  R(K_n)\ge7n+3,
  \qquad
  s(K_n)\le5n+3. \tag{4.1}
\end{equation}
Consequently,
\begin{equation}
  R(K_n)-s(K_n)\ge2n\longrightarrow\infty. \tag{4.2}
\end{equation}
\end{corollary}

\begin{proof}
Proposition~\ref{prop:J} gives $c(J)=6$, $br(J)=2$, and an
eight-stick realization. Substitution into Proposition~\ref{prop:estimate}
gives (4.1) and (4.2).
\end{proof}

The equality $c(K_n)=6n$ shows that the knots $K_n$ are pairwise distinct.
Since $J$ is nontrivial, $K_n$ is a connected sum of two nontrivial knots for
every $n\ge2$, and hence is composite. The nontrivial knots $T_{p-1,p}$ used
by Johnson~\cite{Johnson2012} are prime torus knots. Thus
Corollary~4.1 provides an infinite family distinct
from the original torus-knot family.

\end{document}